\documentclass[11pt,reqno]{amsart}
\usepackage[letterpaper,margin=1.05in,headheight=14pt]{geometry}
\usepackage[T1]{fontenc}
\usepackage{lmodern}
\usepackage{amsmath,amssymb,amsthm,mathrsfs}
\usepackage{enumitem}
\usepackage{microtype}
\usepackage{xcolor}
\usepackage{tikz}
\definecolor{linknavy}{HTML}{17365D}
\definecolor{citegreen}{HTML}{1E5B45}
\usepackage[
colorlinks=true,
linkcolor=linknavy,
citecolor=citegreen,
urlcolor=linknavy,
pdfauthor={Levi Segal and Jacques Verstraete},
pdftitle={Logarithmic circumference in tough graphs},
pdfsubject={A logarithmic cycle bound and its sharpness for toughness at most one},
pdfkeywords={Graph toughness, circumference, treedepth, long cycles}
]{hyperref}
\usepackage[capitalise,noabbrev]{cleveref}
\numberwithin{equation}{section}
\setlist{leftmargin=2.1em,itemsep=0.25em,topsep=0.45em}
\allowdisplaybreaks[2]
\newtheorem{theorem}{Theorem}[section]

\newtheorem{corollary}[theorem]{Corollary}
\newtheorem{lemma}[theorem]{Lemma}
\theoremstyle{definition}

\newtheorem{conjecture}[theorem]{Conjecture}
\newtheorem{problem}[theorem]{Problem}
\theoremstyle{remark}

\crefname{theorem}{Theorem}{Theorems}
\Crefname{theorem}{Theorem}{Theorems}
\crefname{proposition}{Proposition}{Propositions}
\Crefname{proposition}{Proposition}{Propositions}
\crefname{corollary}{Corollary}{Corollaries}
\Crefname{corollary}{Corollary}{Corollaries}
\crefname{lemma}{Lemma}{Lemmas}
\Crefname{lemma}{Lemma}{Lemmas}
\crefname{problem}{Problem}{Problems}
\Crefname{problem}{Problem}{Problems}
\crefname{definition}{Definition}{Definitions}
\Crefname{definition}{Definition}{Definitions}
\crefname{conjecture}{Conjecture}{Conjectures}
\Crefname{conjecture}{Conjecture}{Conjectures}
\crefname{section}{Section}{Sections}
\Crefname{section}{Section}{Sections}
\crefname{equation}{Equation}{Equations}
\Crefname{equation}{Equation}{Equations}

\DeclareMathOperator{\circum}{circ}
\DeclareMathOperator{\td}{td}

\newcommand{\BJM}{Bria\'nski, Joret, Majewski, Micek, Seweryn, and Sharma}
\makeatletter
\def\thm@space@setup{%
  \thm@preskip=12pt plus 2pt minus 2pt
  \thm@postskip=12pt plus 2pt minus 2pt
}
\makeatother
\title[Logarithmic circumference in tough graphs]
{Logarithmic circumference in tough graphs}
\author{Levi Segal}
\address[Levi Segal]{Stanford Online High School}
\email{levisegal0@gmail.com}
\date{}
\author{Jacques Verstraete}
\address[Jacques Verstraete]{Department of Mathematics, University of California, San Diego}
\email{jverstraete@ucsd.edu}
\date{}

\subjclass[2020]{Primary 05C38; Secondary 05C40}
\keywords{Graph toughness, circumference, treedepth, long cycles}

\begin{document}
\begin{abstract}
For every real $t>0$, we prove that every $2$-connected $t$-tough graph contains a cycle of length at least $\ell$ whenever $n \leq \ell(1 + 1/t)^{\lfloor \ell/2\rfloor - 1}$.
This establishes the logarithmic bound conjectured by Broersma, van den Heuvel, Jung, and Veldman. 
\end{abstract}
\maketitle

\begin{samepage}
\section*{AI declaration statement}
We acknowledge the assistance of OpenAI's GPT-5.6 in drafting and revising the manuscript, auditing
proofs, suggesting proof strategies, and recommending relevant literature. In particular,
GPT-5.6 suggested considering trees and elimination trees in the study
of the logarithmic circumference bound, helping to shape the approach
developed here. AI did not give a full proof for any theorem or proposition.
The authors have checked the proofs, modified the presentation, and added further explanations where appropriate. They take full responsibility for all the content in the paper.
\end{samepage}

\section{Introduction}

We consider finite simple undirected graphs. For a graph $G$, let $\omega(G)$ denote the
number of components of $G$, and define the \emph{circumference} of $G$ to be the number of
vertices of a longest cycle in $G$, denoted $\circum(G)$. A graph $G$ is \emph{$t$-tough} if
\begin{equation}\label{eq:toughness}
 |S|\ge t\,\omega(G-S)
 \qquad\text{whenever }\omega(G-S)\ge2.
\end{equation}
The \emph{toughness} $\tau(G)$ is the largest such $t$, with complete graphs assigned infinite toughness. A graph is $2$-connected if it has at least
three vertices, is connected, and remains connected after deleting any
one vertex. Throughout, $\log$ denotes the natural logarithm.

\bigskip

The central questions ask for the relationship between the toughness and the circumference of graphs -- see the survey of Bauer, Broersma, and  Schmeichel~\cite{BBS}. In particular, Chv\'{a}tal conjectured that there exist $t_0$ such that if $G$ is a $t_0$-tough graph, then $G$ is hamiltonian, and this conjecture remains open. In this paper we address the conjecture of Broersma, van den
Heuvel, Jung, and Veldman~\cite{BHJV} (see also Conjecture~2 on page~10 of the survey~\cite{BBS}): 

\begin{conjecture}\label{conj-main}
    For each $t > 0$, there exists $c_t > 0$ such that every 2-connected $t$-tough $n$-vertex graph contains a cycle of length at least $c_t\log n$. 
    \end{conjecture}
The assumption of $2$-connectivity alone gives no lower bound on the
circumference that tends to infinity with the order, as the complete bipartite graph $K_{2,n-2}$ shows. The earlier work on this question is described by Bauer, Broersma, and
Schmeichel~\cite[pp.~10--11]{BBS}. In 1993, Broersma, van den Heuvel, Jung, and Veldman~\cite{BHJV} proved for every fixed $t>0$, every 2-connected $t$-tough $n$-vertex graph contains a cycle of length at least $\ell$ where $\ell \log \ell \ge (2-o(1))\log n$ as $n \rightarrow \infty$. In particular, this bound guarantees cycles of order at least $\log n/\log\log n$. The same authors obtained a lower bound proportional to $\log n$ under $3$-connectivity. B\"ohme, Broersma, and Veldman~\cite{BBV} settled the planar case of Conjecture \ref{conj-main} in
1996. In this paper, we prove Conjecture \ref{conj-main}:

\begin{theorem}\label{thm:main}
Let $t>0$, and $G$ be a $2$-connected
$t$-tough $n$-vertex graph with $\circum(G) = \ell$. Then
\begin{equation}\label{eq:main}
 n \le  \ell(1 + 1/t)^{\lfloor \ell/2\rfloor - 1}.               
\end{equation}
In particular, as $n \rightarrow \infty$,
\begin{equation}\label{eq:main2}
 \circum(G) \ge \frac{2\log n}{\log(1+1/t)} -  \frac{2\log\log n}{\log(1+1/t)} - O_t(1).             
\end{equation}
\end{theorem}

Here the term $O_t(1)$ is bounded independently of $n$, with a bound allowed to depend on $t$. For $t = 1$, a bound valid for all $n$ can be obtained from the proof:

\begin{corollary}\label{cor:one-tough}
Every $1$-tough graph $G$ on $n\ge3$ vertices satisfies
\[
    \circum(G)\ge
    2\log_2 n-2\log_2\log_2 n.
\]
\end{corollary}

The logarithmic dependence on $n$ is best possible for every fixed
$0<t\le1$, as the examples in~\cite{BHJV} already show. A later construction of Kabela~\cite[Proposition~18]{Kabela} gives a better bound: if $G$ is the square of a complete ternary tree (the root has four children and all other non-leaf vertices have three children), then $n = |V(G)| = 2 \cdot 3^h - 1$ and it is straightforward to check that $G$ is 1-tough and $\circum(G) = 6h - 1 \leq (6/\log 3) \cdot \log n$. More generally for $0 < t \leq 1/2$, there exists a $t$-tough graph $G$ such that 
\[ \circum(G) \leq \Big\lceil\frac{2\log(n - 1)}{\log \lfloor 1/t\rfloor}\Big\rceil\]
which is simply obtained from a complete $\lfloor 1/t \rfloor$-ary tree by adding one vertex adjacent to every vertex in the tree. The height of the tree is 
$h = \lceil \log_{\lfloor 1/t \rfloor}(n - 1)\rfloor$ and the circumference of the graph is then $2h + 2$. This determines the order of growth of the circumference for $0<t\le1$, and determines the optimal coefficient which is asymptotic to $2/\log\lfloor 1/t\rfloor$ as $t \rightarrow 0$. The optimal
leading constant and sharpness for $t=1$ are not determined here. 

\begin{problem}
Determine 
\[ \liminf \left\{\frac{\circum(G)}{\log |V(G)|} : \tau(G) = 1\right\}.\]    
\end{problem}

The limit inferior is between $6/\log 3$ and $2/\log 2$.

\section{Treedepth and the lower bound}\label{sec:lower}

An \emph{elimination forest} of $G$ is a rooted forest on $V(G)$ such
that, for every edge $xy\in E(G)$, one of $x,y$ is an ancestor of the
other. The forest need not be a subgraph of $G$. Its height is the
maximum number of vertices on a path from a root to a leaf, and
$\td(G)$ is the minimum height of an elimination forest of $G$. This is called the \emph{treedepth} of $G$. For a connected graph, every elimination forest is a single tree,
because there are no graph edges between its distinct trees.
The argument is to show that $t$-tough graphs have large treedepth. 

\begin{lemma}\label{lem:growth}
Let $t>0$, and let $G$ be a non-complete connected $t$-tough
graph of order $n$. If $T$ is an elimination tree of $G$ of height
$h$, then $h>\lceil 2t\rceil$ and
\begin{equation}\label{eq:order-height}
 n\le \lceil 2t\rceil(1+1/t)^{h-\lceil 2t\rceil}.
\end{equation}
\end{lemma}

\begin{proof}
Give the root of $T$ depth one. For $j\le h$, let $L_j$ be the set of vertices at depth $j$, and put
\[
 P_j=L_1\cup\cdots\cup L_j.
\]
For $j<h$, the descendant subtrees rooted at the vertices of
$L_{j+1}$ partition $V(G)\setminus P_j$. Vertices in different such
subtrees are incomparable in the ancestor order, so no graph edge
joins two distinct subtrees. Each subtree is nonempty, so
\begin{equation}\label{eq:components}
 \omega(G-P_j)\ge |L_{j+1}|.
\end{equation}
Whenever $|L_{j+1}|\ge2$, we may apply toughness to $P_j$ to obtain
\begin{equation}\label{eq:level-bound}
 t|L_{j+1}|\le |P_j|.
\end{equation}

We show that the first $q = \lceil 2t\rceil$ levels are singletons. The first
level consists of the root. Suppose $1\le j<q$ and the first $j$
levels are singletons, so $|P_j|=j$. Since $G$ is not complete, $n >q$, so there is a further
level. If it contained at least two vertices, then
\eqref{eq:level-bound} would imply
\[
 j=|P_j|\ge2t,
\]
contrary to $j\le q-1<2t$. Induction gives $p_q=q$. For $q=1$
this just says that the first level is the root. Since $n>q$, the
tree continues beyond level $q$, and $h>q$.

For $q\le j<h$, we have $p_j\ge q\ge2t\ge t$. Thus
$|L_{j+1}|\le p_j/t$ also holds when $|L_{j+1}|=1$. Combining the
two cases gives
\begin{equation}\label{eq:recurrence}
 p_{j+1}=p_j+|L_{j+1}|\le(1+1/t)p_j
 \qquad(q\le j<h).
\end{equation}
Iterating from $p_q=q$ to $p_h=n$ proves~\eqref{eq:order-height}.
\end{proof}

The following consequence of~\cite[Theorem~2]{BJM} connects treedepth to circumference: 

\begin{theorem}[\BJM]\label{thm:td-circ}
Every $2$-connected graph $G$ satisfies $\td(G)\le\circum(G)$.
\end{theorem}

Combined with Theorem \ref{thm:td-circ}, the above lemma gives a short proof of Conjecture \ref{conj-main} which is weaker than Theorem \ref{thm:main} (\ref{eq:main2}) by an asymptotic factor of two:

\begin{theorem}\label{thm:main2}
Let $t>0$, and $G$ be a non-complete $2$-connected
$t$-tough $n$-vertex graph. Then
\begin{equation}
 \circum(G) \ge \lceil 2t \rceil + \frac{\log n/\lceil 2t\rceil}{\log(1+1/t)}.             
\end{equation}
\end{theorem}

\begin{proof}
Choose an elimination tree of height
$h=\td(G)$. Taking logarithms in Lemma~\ref{lem:growth} yields
\[
 h\ge q+\frac{\log(n/q)}{\log(1+1/t)}.
\]
Since $h$ is an integer and $h\le\circum(G)$ by
Theorem~\ref{thm:td-circ}, we are done.
\end{proof}

\bigskip

\section{Proof of Theorem \ref{thm:main}}

The lower bound on circumference in Theorem \ref{thm:main2} $\log n/\log(1+1/t) - O_t(1)$, whereas the bound in Theorem \ref{thm:main} (\ref{eq:main2}) is asymptotically twice as large as $n \rightarrow \infty$. To obtain this improvement, we first remove a longest cycle $C$ from the graph $G$, say of length $\ell$, and argue the treedepth of the remaining graph is at most $\ell/2$. Then an argument similar to the proof of Lemma \ref{lem:growth} applied to an elimination forest of height $h$ in $G - V(C)$ shows $n \leq \ell (1 + 1/t)^{h}$ whence we obtain Theorem \ref{thm:main}.

\bigskip

The following lemma is implicit in the proof
of Lemma~4 of Bria\'nski et al.~\cite{BJM}.

\begin{lemma}\label{lem:path}
Let $F$ be a $2$-connected graph, and let $a,b\in V(F)$ be distinct.
Then $F$ contains an $a$--$b$ path $P$ such that
\[
    |E(P)|\ge \td(F-\{a,b\})+1.
\]
\end{lemma}

\begin{proof}
In the proof of \cite[Lemma~4]{BJM}, the authors construct
paths through a sequence of blocks of $F-b$, starting at $a$,
whose total length is at least $\td(F-\{a,b\})$.
They concatenate these paths and then append an edge into $b$.
The resulting $a$--$b$ path therefore has length at least
$\td(F-\{a,b\})+1$, as claimed.
\end{proof}

\begin{lemma}\label{lem:delete-cycle}
Let $G$ be a $2$-connected graph, and let $C$ be a longest cycle
of length $\ell$. Then
\begin{equation}\label{eq:half-depth}
    \td(G-V(C))\le \left\lfloor\frac\ell2\right\rfloor-1.
\end{equation}
\end{lemma}

\begin{proof}
The assertion is immediate if $C$ is Hamiltonian. Otherwise, let $H$
be a component of $G-V(C)$. Since $G$ is $2$-connected, $H$ has
neighbors at at least two distinct vertices of $C$.
Choose one such vertex $a\in V(C)$. Construct an auxiliary graph $F$ by adding a vertex $b$ to the induced subgraph $G[V(H) \cup\{a\}]$ by adding all edges between $b$ and every vertex of $H$ having a neighbor in $V(C)\setminus\{a\}$, as well as the edge $ab$. It is straightforward to see that $F$ is $2$-connected. 

\bigskip

Since $F-\{a,b\}=H$, Lemma~\ref{lem:path} gives an $a$-$b$ path
$P$ of length at least $\td(H)+1$. This length is at least two,
so $P$ does not use the edge $ab$.
Replace its final edge into $b$ by an edge to a corresponding vertex
$z\in V(C)\setminus\{a\}$. We obtain an $a$-$z$ path in $G$
with all internal vertices in $H$ and length at least $\td(H)+1$.
Together with a longest $a$-$z$ path in $C$, this path forms a cycle
of length at least
\[
    \td(H)+1+\left\lceil\frac\ell2\right\rceil \le\ell,
\]
by the maximality of $C$, and hence 
$\td(H)\le\lfloor\ell/2\rfloor-1$.
Taking the maximum over the components $H$ of $G-V(C)$ proves
\eqref{eq:half-depth}.
\end{proof}

\bigskip

\begin{proof}[Proof of Theorem~\ref{thm:main}]
Write $\ell=\circum(G)$, and let $C$ be a cycle of length $\ell$.
We first prove \eqref{eq:main}.
If $\ell=n$, this inequality is immediate, since $\ell\ge3$.
Assume therefore that $\ell<n$, so $G$ is not complete.

We first note that $\ell>t$. Indeed, a vertex $v$ of minimum degree
has a nonneighbor, since $G$ is not complete. Deleting $N_G(v)$
leaves $v$ isolated and at least one other vertex, so toughness gives
\[
    \delta(G)=|N_G(v)|\ge2t.
\]
An endpoint of a longest path has all its neighbors on the path.
The segment joining that endpoint to its farthest neighbor,
together with their joining edge, gives a cycle of length at least
$\delta(G)+1$. Consequently,
\begin{equation}\label{eq:ell-large}
    \ell\ge\delta(G)+1>t.
\end{equation}

By Lemma~\ref{lem:delete-cycle}, the graph $G-V(C)$ has an
elimination forest of height
\[
    h\le\left\lfloor\frac\ell2\right\rfloor-1.
\]
For $1\le j\le h$, let $L_j$ be the set of vertices at depth $j$,
with roots at depth one. Put
\[
    S_0=V(C),\qquad
    S_j=V(C)\cup L_1\cup\cdots\cup L_j,
    \qquad s_j=|S_j|.
\]
Thus $s_0=\ell$ and $s_h=n$.

\bigskip

For $0\le j<h$, the descendant subtrees rooted at the vertices of
$L_{j+1}$ partition $V(G)\setminus S_j$.
Vertices in different such subtrees are incomparable in the ancestor
order. Hence no graph edge joins two distinct subtrees, and
\begin{equation}\label{eq:components}
    \omega(G-S_j)\ge |L_{j+1}|.
\end{equation}
If $|L_{j+1}|\ge2$, toughness and \eqref{eq:components} imply
\[
    t|L_{j+1}|\le s_j.
\]
If $|L_{j+1}|=1$, the same inequality follows from
$s_j\ge\ell>t$, by \eqref{eq:ell-large}.
Therefore, in either case,
\[
    s_{j+1}
    =s_j+|L_{j+1}|
    \le\left(1+\frac1t\right)s_j.
\]
Iterating gives
\[
    n=s_h
    \le\ell\left(1+\frac1t\right)^h
    \le\ell\left(1+\frac1t\right)^{\lfloor\ell/2\rfloor-1},
\]
which is \eqref{eq:main}.
Notice that toughness has only been applied to $G$, not to
$G-V(C)$.

\bigskip

To deduce \eqref{eq:main2}, put
$\lambda=\log(1+1/t)>0$.
Taking logarithms in \eqref{eq:main} gives
\[
    \log n
    \le\log\ell+
       \left(\left\lfloor\frac\ell2\right\rfloor-1\right)\lambda
    \le\log\ell+\left(\frac\ell2-1\right)\lambda.
\]
Consequently,
\begin{equation}\label{eq:rearranged}
    \ell\ge
    2+\frac{2}{\lambda}\bigl(\log n-\log\ell\bigr).
\end{equation}
If $\ell\ge(2/\lambda)\log n$, then
\eqref{eq:main2} follows immediately.
Otherwise,
\[
    \log\ell
    \le\log\log n+\log\frac2\lambda.
\]
Substituting into \eqref{eq:rearranged}, we obtain
\[
    \ell\ge
    \frac{2}{\lambda}\bigl(\log n-\log\log n\bigr)
    +2-\frac{2}{\lambda}\log\frac2\lambda.
\]
Since $t$ is fixed, the final two terms depend only on $t$.
This proves \eqref{eq:main2}.
\end{proof}

\bigskip
\begin{proof}[Proof of Corollary~\ref{cor:one-tough}]
A $1$-tough graph on at least three vertices is $2$-connected.
Set $\ell=\circum(G)$.
Applying \eqref{eq:main} with $t=1$ gives
\[
    n\le\ell\,2^{\lfloor\ell/2\rfloor-1}
      \le\frac\ell2\,2^{\ell/2}.
\]
Put $x=\ell/2$ and $L=\log_2 n$. Then
\[
    L\le x+\log_2 x.
\]
Since $L>1$, if $x\ge L$ then $x\ge L-\log_2 L$.
If $x<L$, the preceding inequality gives
\[
    x\ge L-\log_2 x\ge L-\log_2 L.
\]
Thus in both cases $\ell=2x\ge2L-2\log_2 L$, as required.
\end{proof}

\end{document}